\documentclass[12pt]{amsart}

\usepackage{color,fancyhdr}
\usepackage[all]{xy}

\usepackage{amsmath,amsthm,mathrsfs,amsfonts,
amssymb,times,latexsym,mathabx}

\usepackage[usenames,dvipsnames,svgnames,table]{xcolor}

\DeclareMathAlphabet{\curly}{U}{rsfs}{m}{n}  

\newtheorem{theorem}{Theorem}[section]
\newtheorem{lemma}{Lemma}[section]
\newtheorem{proposition}[lemma]{Proposition}
\theoremstyle{definition}

\newtheorem{corollary}[theorem]{Corollary}

\theoremstyle{problem}
\newtheorem{problem}[theorem]{Problem}
\usepackage{xcolor}

\numberwithin{equation}{section}

\makeatletter
\renewcommand{\pmod}[1]{\allowbreak\mkern7mu({\operator@font mod}\,\,#1)}
\makeatother

\newcommand{\be}{\begin{equation}}
\newcommand{\ee}{\end{equation}}

\renewcommand{\le}{\leqslant}
\renewcommand{\leq}{\leqslant}
\renewcommand{\ge}{\geqslant}
\renewcommand{\geq}{\geqslant}

\begin{document}

\title[On S\'ark\"ozy-S\'{o}s Theorem related to representation functions]
{On S\'ark\"ozy-S\'{o}s Theorem related to representation functions}

\author{Jin-Hui Fang}
\address{School of Mathematical Sciences, Nanjing Normal University, Nanjing, Jiangsu, PR China}
\email{fangjinhui1114@163.com}

\author[Kiss]{S\'andor Z. Kiss}
\address{Department of Algebra and Geometry, Institute of Mathemetics, Budapest University of Technology and Economics, M\H{u}egyetem rkp. 3., H-1111 Budapest, Hungary; 
\newline \hspace*{4mm}
HUN-REN Alfr\'ed R\'enyi Institute of Mathematics, Re\'altanoda utca 13--15., H-1053 Budapest, Hungary}
\email{kiss.sandor@ttk.bme.hu}

\author{Wei Niu}
\address{School of Mathematics and Statistics, Xi'an Jiaotong University, Xi'an, Shaanxi, PR China}
\email{wei.niu@stu.xjtu.edu.cn}

\author[S\'andor]{Csaba S\'andor}
\address{Department of Stochastics, Institute of Mathemetics, Budapest University of Technology and Economics, M\H{u}egyetem rkp. 3., H-1111 Budapest, Hungary; \newline \hspace*{4mm}
HUN-REN Alfr\'ed R\'enyi Institute of Mathematics, Re\'altanoda utca 13--15., H-1053 Budapest,  Hungary; \newline \hspace*{4mm} 
MTA--HUN-REN RI Lend\"ulet ``Momentum'' Arithmetic Combinatorics Research Group, Re\'altanoda utca 13--15., H-1053 Budapest,  Hungary}
\email{sandor.csaba@ttk.bme.hu}

\thanks{The first author is supported by  the National Natural
Science Foundation of China, Grant No. 12571005 and Basic Research Program of Jiangsu Province, Grant No. BK20250139. The second author is supported by the NKFIH grants K146387, KKP 144059. The second author is also supported by the National Research, Development and Innovation Office NKFIH (Excellence program, Grant Nr. 153829). The third author is supported from the CSC programs. The forth author is supported by the NKFIH Grants No. K129335, KKP 144059 and the Lend\"ulet ''Momentum'' program of the Hungarian Academy of Sciences (MTA)}
\date{\today}%

\begin{abstract}
Let $\mathbb{N}_0$ be the set of all nonnegative integers.
For a nonempty set $\mathcal{A}\subseteq \mathbb{N}_0$ and integers $n,h\ge 2$, let $r_{h}(\mathcal{A},n)$ be the number of representations of $n$ as
$a_1+\cdots+a_h$, where $a_1\le \cdots\le a_h$ and $a_i\in \mathcal{A}$ for $i=1,\cdots,h$.
In 2016, Chen and Tang showed that, for any given distinct positive integers $u_1,\cdots,u_k$ and positive rational numbers 
$\alpha_1,\cdots,\alpha_k$ with $\alpha_1+\cdots+\alpha_k=1$,
there are infinitely many sets $\mathcal{A}\subseteq \mathbb{N}_0$ such that $r_{h}(\mathcal{A},n)\ge 1$ for all nonnegative integers $n$ and the set of $n$ with $r_{h}(\mathcal{A},n)=u_i$ has density $\alpha_i$ for all integer $i=1,\cdots,k$. In this paper, we consider the irrational numbers $\alpha_i$ as well. As a main result, we prove that, 
for any nonnegative numbers
$\alpha_0,\cdots,\alpha_m$ with $\alpha_0+\cdots+\alpha_m=1$,
there are infinitely many sets $\mathcal{A}\subseteq \mathbb{N}_0$ such that the set of $n$ with $r_{2}(\mathcal{A},n)=i$ has density $\alpha_i$ for all integer $i=0,\cdots,m$. Other related results are also contained.
\end{abstract}

\maketitle


\section{\bf Introduction}

Let $\mathbb{N}_0$ be the set of all nonnegative integers and $\mathbb{N}$ be the set of all positive integers.
For a nonempty set $\mathcal{A}\subseteq \mathbb{N}_0$ and integers $n,h\ge 2$, let $r_{h}(\mathcal{A},n)$ be the number of representations of $n$ as
$a_1+\cdots+a_h$, where $a_1\le \cdots\le a_h$ and $a_i\in \mathcal{A}$ for $i=1,\cdots,h$. For $u\in \mathbb{N}_0$ and $N\in \mathbb{N}$, define
\begin{eqnarray*}
\mathcal{S}_u^{(h)}(\mathcal{A})=\#\{n\in \mathbb{N}:r_{h}(\mathcal{A},n)=u\}
\end{eqnarray*}
and
\begin{eqnarray*}
\mathcal{S}_u^{(h)}(\mathcal{A},N)=\#\{n\le N:r_{h}(\mathcal{A},n)=u\}.
\end{eqnarray*}
For the sets $\mathcal{A}$, $\mathcal{B}$ of integers, we denote the sumset by
\begin{eqnarray*}
\mathcal{A} + \mathcal{B} = \{a+b: a\in \mathcal{A}, b\in \mathcal{B}\}.
\end{eqnarray*}
For a real number $x$, let $\mathcal{A}(x)$ be the number of positive integers in $\mathcal{A}$ not exceeding $x$.\vskip2mm

In 1997, S\'ark\"ozy and S\'{o}s \cite{Sarkozy1997} showed that for every finite set $U$ of positive integers there is a set $\mathcal{A}$ such that, apart from a ``thin" set of integers $n$, $r_{2}(\mathcal{A},n)$
assumes only the prescribed values $u\in U$ with about the same frequency. In fact, they proved the following nice result:\vskip2mm

\noindent{\textbf{Theorem A (\cite[Theorem 4.3]{Sarkozy1997}).}}
Let $k\in \mathbb{N}$ and let $u_1<u_2<\cdots<u_k$ be positive integers. Then there is an infinite set $\mathcal{A}\subseteq \mathbb{N}_0$
such that writing
\begin{eqnarray*}
\mathcal{B}=\mathbb{N}\setminus (\bigcup_{i=1}^k \mathcal{S}_{u_i}^{(2)}(\mathcal{A}))
\end{eqnarray*}
we have
\begin{eqnarray*}
\mathcal{S}_{u_i}^{(2)}(\mathcal{A},N)=\frac{N}{k}+O(N^{\alpha})
\end{eqnarray*}
and
\begin{eqnarray*}
\mathcal{B}(N)=O(N^{\alpha}), \hskip4mm\mbox{where}\hskip3mm \alpha=\frac{\log 3}{\log 4}.
\end{eqnarray*}
\vskip2mm

In 2016, Chen and Tang \cite{Chen2016} extended the above S\'ark\"ozy-S\'{o}s Theorem to general integer $h\ge 2$, where the method is different.
\vskip2mm

\noindent{\textbf{Theorem B (\cite[Theorem 1]{Chen2016}).}}
Let $k,h\in \mathbb{N}$ with $h\ge 2$ and let $u_1<u_2<\cdots<u_k$ be positive integers. Let $\alpha_i$ ($1\le i\le k$) be positive rational numbers with $\alpha_1+\cdots+\alpha_k=1$. Then there are infinitely many bases $\mathcal{A}$ of order $h$ such that
\begin{eqnarray}\label{1}
\mathcal{S}_{u_i}^{(h)}(\mathcal{A},N)=\alpha_i N+O(N^{\alpha}),\hskip4mm 1\le i\le k,
\end{eqnarray}
where $\alpha=\alpha(\mathcal{A})$ with $0<\alpha<1$.
\vskip1mm

Let $\mathcal{B}=\mathbb{N}\setminus (\bigcup_{i=1}^k \mathcal{S}_{u_i}^{(h)}(\mathcal{A}))$. If \eqref{1} holds, then $B(N)=O(N^{\alpha})$.
\vskip2mm

Until now there is no progress on the \emph{irrational} case. Recently the authors consider the irrational case and obtain the following result:

\begin{theorem}\label{thm1}
Let $m$ be any given positive integer and let $\alpha_0,
\cdots,\alpha_m$ be nonnegative real numbers satisfying $\alpha_0+\cdots+\alpha_m=1$. Then there are infinitely many sets 
$ \mathcal{A}\subseteq \mathbb{N}_0$ such that for every integer $i$ with $0\le i
\le m$ we have
\begin{eqnarray}\label{2}
\mathcal{S}_{i}^{(2)}(\mathcal{A},N)=\alpha_i N+O(N^{\frac{3}{4}}).
\end{eqnarray}
\end{theorem}
\vskip3mm


Let $k\in \mathbb {N}_0$ and let $u_1<\cdots<u_k$ be positive integers. Let $\beta_1,
\cdots,\beta_k$ be positive real numbers satisfying $\beta_1+\cdots+\beta_k=1$. Let $m=u_k$ and 
\begin{eqnarray*}
\alpha_{j}=
\begin{cases} 0, &\mbox{if}\hskip3mm 0\le j\le u_k \hskip3mm\mbox{and}\hskip3mm j\neq u_1,\cdots, u_k 
\cr \beta_i, &\mbox{if}\hskip3mm j=u_i.\end{cases}
\end{eqnarray*}
Then we obtain the following result from the Theorem \ref{thm1}.
\begin{corollary}
Let $k\in \mathbb {N}_0$ and let $u_1<\cdots<u_k$ be positive integers. Let $\beta_1,
\cdots,\beta_k$ be positive real numbers satisfying $\beta_1+\cdots+\beta_k=1$. Then there are infinitely many sets $\mathcal{A}\subseteq \mathbb{N}_0$ such that 
\begin{align*}
\mathcal{S}_{u_i}^{(2)}(\mathcal{A},N)=\beta_i N+O(N^{\frac{3}{4}}).
\end{align*} 
\end{corollary}
\newpage

Finally, we pose two open problems for further research.
\begin{problem}
Let $\alpha_0,\alpha _1,
\cdots $ be nonnegative real numbers satisfying $\displaystyle \sum_{i=0}^{\infty}\alpha _i=1$. Does there exist a set $\mathcal{A}\subseteq \mathbb{N}_0$ such that for every nonnegative integer $i$ we have
\begin{eqnarray*}
\mathcal{S}_{i}^{(2)}(\mathcal{A},N)=\alpha_i N+O_i(N^{\frac{3}{4}})?
\end{eqnarray*}
\end{problem}
\begin{problem}
Let $m$ be any given positive integer and let $\alpha_0,
\cdots,\alpha_m$ be nonnegative real numbers satisfying $\alpha_0+\cdots+\alpha_m=1$. Does there exists a number $c>0$ such that there is no set $\mathcal{A}\subseteq \mathbb{N}_0$ with
\begin{eqnarray*}
\mathcal{S}_{i}^{(2)}(\mathcal{A},N)=\alpha _i N+O(N^c)\quad ?
\end{eqnarray*}
\end{problem}

\section{\bf Preliminary Lemma}
We define the sets
\begin{eqnarray*}
F = \left \{\sum_{i=0}^{\infty}\varepsilon_{i}9^{i}: \varepsilon_{i}\in \{0,1,2\}, \varepsilon_{i} = 0 \textnormal{ all but finitely many } i \right \}
\hskip4mm\mbox{and}\hskip4mm G = 3\times F. 
\end{eqnarray*}
Then every nonnegative integer $n$ can be uniquely written in the form $n = f+g$, where $f\in F$, $g\in G$. Abusing the notation, we will write $f = f(n)$.
We define the height of 
\begin{eqnarray*}
n = \sum_{i=0}^{\infty}\delta_{i}3^{i} 
\end{eqnarray*}
 where $\delta_{i}\in \{0,1,2\}$ and $\delta_{i} = 0 \textnormal{ all but finitely many } i$ as
\begin{eqnarray*}
H(n) = \max\{i: \delta_{i}\neq 0 \} \hskip4mm\mbox{and}\hskip4mm H(0)=-1.
\end{eqnarray*}
\vskip2mm

\begin{lemma}
Let $\alpha_{0}, \dots{} ,\alpha_{m} \ge 0$ be real numbers with 
\begin{eqnarray*}
\sum_{i=0}^{m}\alpha_{i} = 1.
\end{eqnarray*}
Then there exists a function $\tau: F\rightarrow \{0,1,\dots{} m\}$ such that for every $0 \le i \le m$, one has
\begin{equation}
   \#\{0 \le n \le N: \tau(f(n)) = i\} = \alpha_{i}N + O(N^{3/4}), 
\end{equation}
\begin{equation}
   \#\{0 \le n \le N: \tau(f(n)) \neq \tau(f(n-1))\} = O(N^{3/4}) 
\end{equation}
\end{lemma}

\begin{proof}
In the first step we prove that for every $s \ge 1$ there exists a map 
$\chi_{2s}: \{0,1,2\}^{2s}\rightarrow \{0,1,\dots{} ,m\}$ such that
\begin{eqnarray*}
\#\{(u_1,\dots{} ,u_{2s})\in \{0,1,2\}^{2s}: \chi_{2s}(u_1,\dots{} ,u_{2s}) = i\} = \theta_{2s,i}3^{2s}
\end{eqnarray*}
with $|\theta_{2s,i}3^{2s}-\alpha_{i}3^{2s}|\le 1$ for every $0\leq i\leq m$.
Obviously, $3^{2s} - m \le \sum_{i=0}^{m}\lfloor \alpha_{i}3^{2s}\rfloor \le 3^{2s}$ and if $3^{2s} - u = \sum_{i=0}^{m}\lfloor \alpha_{i}3^{2s}\rfloor$, where $0 \le u \le m$, then let
\[
\theta_{2s,i} =
\begin{cases}
    \frac{\lfloor \alpha_{i}3^{2s}\rfloor+1}{3^{2s}} & \text{if }\hskip3mm 0 \le i \le u-1, \\
    \frac{\lfloor \alpha_{i}3^{2s}\rfloor}{3^{2s}} & \text{if }\hskip3mm u \le i \le m.
\end{cases}
\]
For an $f\in F$ with $f = \sum_{i=0}^{\infty}\delta_{i}3^{i}$, 
$\delta_{i}\in \{0,1,2\}$ we define
\begin{eqnarray*}
\tau(f) =
\begin{cases}
               0, & \text{if }  0 \le f < 3^{8} \\

    \chi_{2s}(\delta_{2s}, \delta_{2s+2}, \dots{} ,\delta_{6s-2}), & \text{if } 3^{8s} \le f < 3^{8s+8}, s \ge 1.
\end{cases}
\end{eqnarray*}
\vskip3mm

We show that $\tau$ is suitable. For any $s \ge 0$ and $N\in \mathbb{Z}^+$, define 
\begin{eqnarray*}
E_{8s}&=&\{n\in \mathbb N: 8s \le H(f(n)) < 8s+8\},\\
E_{8s}(N)&=&\{1 \le n \le N: 8s \le H(f(n)) < 8s+8\}. 
\end{eqnarray*}
\vskip4mm

\begin{proposition}
    Let $S$ be a nonnegative integer such that $3^{8S} \le N < 3^{8S+8}$. Then for any $1 \le s \le S$, we have
\begin{eqnarray*}
    \#\{1\le n \le N: n\in E_{8s}, \tau(f(n)) = i\} = \alpha_{i}|E_{8s}(N)|+O(3^{4S+2s}).
\end{eqnarray*}
\end{proposition}

\begin{proof}
For any $s\geq 0 $, partition $\mathbb{N}_0$ into the intervals 
\begin{eqnarray*}
I_{a,8s} = [a3^{8s}, (a+1)3^{8s}[
\end{eqnarray*}
with $a=0,1,2\dots$ We claim that either $I_{a,8s} \subseteq E_{8s}$ or  $I_{a,8s} \cap E_{8s} = \emptyset$ for any $a$.  For $n\neq n^{\prime} \in I_{a,8s}$, we can write 
\begin{eqnarray*}
n=\sum_{i=0}^{\infty}\delta_{i}3^{i}=a3^{8s}+b \ \ \text{and} \ \ n^{\prime}=\sum_{i=0}^{\infty}\delta_{i}^{\prime}3^{i}=a3^{8s}+b^{\prime},
\end{eqnarray*}
with $0\leq b\neq b^{\prime}<3^{8s}$ and $\delta_{i},\delta_{i}^{\prime}\in \{0,1,2\}$. $n\in E_{8s}$ if and only if there exists an integer $v$ satisfying 
\begin{eqnarray*}
4s \leq v< 4s+4
\end{eqnarray*}
such that $\delta_{2v} \neq 0$, but 
$\delta_{2i} = 0$ for $i>v$. Since $b$ can only influence the value of $\delta_{0},\delta_{1},\dots,\delta_{8s-1}$, the existence of such $v$ is determined by $a3^{8s}$.
Moreover, we have
\begin{eqnarray*}
\delta_{i}=\delta_{i}^{\prime} \ \ \text{for} \ \ i\geq 8s,
\end{eqnarray*}
which implies that if
$n\in(\notin)E_{8s},$ then so does $n^{\prime}$.
This proves the claim.
\vskip2mm	
    

For $s \ge 1$, let 
\begin{eqnarray*}
A_{8s}(N) = \{a\ge 1: I_{a,8s}\subseteq E_{8s}(N)\}.
\end{eqnarray*}
We show that for every $a\in A_{8s}(N)$, one has
\begin{eqnarray*}
\#\{n: n\in I_{a,8s}, \tau(f(n)) = i\} = \theta_{2s,i}3^{8s}.
\end{eqnarray*}

 For $0 \le j < 3^{8s}$, let 
\begin{eqnarray*}
a3^{8s} + j = \sum_{i=0}^{8S+7}\delta_{i}^{(j)}3^{i}, \hskip4mm\mbox{where}\hskip3mm \delta_{i}^{(j)}\in \{0,1,2\}. 
\end{eqnarray*}
Obviously,
\begin{eqnarray*}
\{(\delta_{0}^{(j)}, \dots{} ,\delta_{8s-1}^{(j)}): 0 \le j < 3^{8s}\} = \{(\gamma_{1}, \dots{} ,\gamma_{8s}): \gamma_{i}\in \{0,1,2\}\}.
\end{eqnarray*}
Thus, for every $(\eta_{1}, \dots{} ,\eta_{2s})
\in \{0,1,2\}^{2s}$, one has
\begin{eqnarray*}
\#\{0\le j \le 3^{8s}-1: (\delta_{2s}^{(j)}, \delta_{2s+2}^{(j)} \dots{} ,\delta_{6s-2}^{(j)}) = (\eta_{1}, \dots{} ,\eta_{2s})\} = 3^{6s}.
\end{eqnarray*}
Moreover, for any $n\in I_{a,8s}$, noting that $a\in A_{8s}(N)$, we have $n\in E_{8s}(N)$. Therefore, 
\begin{eqnarray*}8s \le H(f(n)) < 8s+8.\end{eqnarray*}
This implies that
\begin{eqnarray*}
\tau(f(n)) =
    \chi_{2s}(\delta_{2s}, \delta_{2s+2}, \dots{} ,\delta_{6s-2}). 
\end{eqnarray*}
Thus,
\begin{eqnarray*}
\#\{n: n\in I_{a,8s}, \tau(f(n)) = i\} 
&=&\#\{(\eta_{1}, \dots{} ,\eta_{2s})
\in \{0,1,2\}^{2s}: \chi_{2s}(\eta_{1}, \dots{} ,\eta_{2s}) = i\} \cdot 3^{6s} \\
&=&\theta_{2s,i}3^{8s}.
\end{eqnarray*}
Recall that $I_{a,8s} = [a3^{8s}, (a+1)3^{8s}[$, and as shown above,  $I_{a,8s}\subseteq E_{8s}$ if and only if $a3^{8s}\in E_{8s}$. This implies that 
\begin{eqnarray*}
|A_{8s}(N)| = \#\{a: a\le \frac{N}{3^{8s}}, a\cdot 3^{8s}\in  E_{8s}(N)\}+O(1).
\end{eqnarray*}
If $a = \sum_{i=0}^{8S+7-8s}\delta_{i}3^{i}$ and $a\cdot 3^{8s}\in  E_{8s}(N)$,  then $\delta_{8} = 0, \delta_{10} = 0, \dots{},\delta_{8S+6-8s}=0$, which implies that
\[
|A_{8s}(N)| \le 3^{8S+7-8s-(4S-4s)} + O(1) = O(3^{4S-4s}). 
\]
Clearly, $\theta_{2s,i} = \alpha_{i} + O(3^{-2s})$.
If $\left \lfloor \frac{N}{3^{8s}} \right \rfloor \cdot 3^{8s}\notin E_{8s}(N)$, then $|E_{8s}(N)| = |A_{8s}(N)|\cdot 3^{8s}$ and
\begin{eqnarray*}
\#\{n\le N: n \in E_{8s}(N), \tau(f(n)) = i\} &=& |A_{8s}(N)|\cdot \theta_{2s,i}3^{8s}
= (\alpha_{i} + O(3^{-2s}))|A_{8s}(N)|\cdot 3^{8s}\\
&=& \alpha_{i} |A_{8s}(N)|\cdot 3^{8s} + O(3^{-2s}\cdot 3^{4S-4s}\cdot 3^{8s}))\\
&=&\alpha_{i}|E_{8s}(N)| + O(3^{4S+2s}).
\end{eqnarray*}
If $\left \lfloor \frac{N}{3^{8s}} \right \rfloor \cdot 3^{8s} \in E_{8s}(N)$, $N = \left \lfloor \frac{N}{3^{8s}} \right \rfloor \cdot 3^{8s} + 3^{8s}-1$, then similarly as in the previous case we have
\begin{eqnarray*}
\#\{n\le N: n \in E_{8s}(N), \tau(f(n)) = i\} &=& |A_{8s}(N)|\cdot \theta_{2s,i}3^{8s}\\
&=& \alpha_{i}|E_{8s}(N)| + O(3^{4S+2s}).
\end{eqnarray*}
If $\left \lfloor \frac{N}{3^{8s}} \right \rfloor \cdot 3^{8s} \in E_{8s}(N)$, $N < \left \lfloor \frac{N}{3^{8s}} \right \rfloor \cdot 3^{8s} + 3^{8s}-1$, then similarly as in the previous
\begin{eqnarray*}
&&\#\{n\le N: n \in E_{8s}(N), \tau(f(n)) = i\} \\
&=& \#\left \{n:n\le \left \lfloor \frac{N}{3^{8s}} \right \rfloor \cdot 3^{8s}-1, \tau(f(n)) = i\right \}\\
&&+\#\left \{n:\left \lfloor \frac{N}{3^{8s}}\right \rfloor \cdot 3^{8s} \le n \le \left \lfloor \frac{N}{3^{8s}}\right \rfloor \cdot 3^{8s}+\left \lfloor \frac{N-\left \lfloor \frac{N}{3^{8s}}\right \rfloor \cdot 3^{8s}}{3^{6s}} \right \rfloor \cdot 3^{6s} - 1, \tau(f(n)) = i \right\}\\
&&+\#\left \{n: \left \lfloor \frac{N}{3^{8s}}\right \rfloor \cdot 3^{8s}+\left \lfloor \frac{N-\left \lfloor \frac{N}{3^{8s}}\right \rfloor \cdot 3^{8s}}{3^{6s}} \right \rfloor \cdot 3^{6s} \le n \le N, \tau(f(n)) = i\right \},
\end{eqnarray*}
where 
\begin{eqnarray*}
\#\left \{n:n\le \left \lfloor \frac{N}{3^{8s}} \right \rfloor \cdot 3^{8s}-1: \tau(f(n)) = i\right \} = |A_{8s}(N)|\cdot \theta_{2s,i}3^{8s}.
\end{eqnarray*}
Let $d$ be a nonnegative integer which satisfies $a3^{8s} + (d+1)3^{6s}-1 \le N$. Then similarly as in the previous argument we get that 
\begin{eqnarray*}
\#\{n: a3^{8s} + d3^{6s}\le n \le a3^{8s} + (d+1)3^{6s}-1, \tau(f(n)) = i\} = \theta_{2s,i}3^{6s}.
\end{eqnarray*}
Since 
\begin{eqnarray*}
\tau(f(n)) =
    \chi_{2s}(\delta_{2s}, \delta_{2s+2}, \dots{} ,\delta_{6s-2}),
\end{eqnarray*}
we have
\begin{eqnarray*}
&&\#\left \{n:\left \lfloor \frac{N}{3^{8s}}\right \rfloor \cdot 3^{8s} \le n \le \left \lfloor \frac{N}{3^{8s}}\right \rfloor \cdot 3^{8s}+\left \lfloor \frac{N-\left \lfloor \frac{N}{3^{8s}}\right \rfloor \cdot 3^{8s}}{3^{6s}} \right \rfloor \cdot 3^{6s} - 1, \tau(f(n)) = i \right\} \\
&&=\left \lfloor \frac{N-\left \lfloor \frac{N}{3^{8s}}\right \rfloor \cdot 3^{8s}}{3^{6s}} \right \rfloor \cdot \theta_{2s,i} \cdot 3^{6s}.
\end{eqnarray*}
Obviously, 
\begin{eqnarray*}
0 \le N-\left(\left \lfloor \frac{N}{3^{8s}}\right \rfloor \cdot 3^{8s}-\left \lfloor \frac{N-\left \lfloor \frac{N}{3^{8s}}\right \rfloor \cdot 3^{8s}}{3^{6s}} \right \rfloor \cdot 3^{6s}\right) \le 3^{6s}.
\end{eqnarray*}
So 
\begin{eqnarray*}
&&\#\left\{n: \left \lfloor \frac{N}{3^{8s}}\right \rfloor \cdot 3^{8s}+\left \lfloor \frac{N-\left \lfloor \frac{N}{3^{8s}}\right \rfloor \cdot 3^{8s}}{3^{6s}} \right \rfloor \cdot 3^{6s} \le n \le N, \tau(f(n)) = i\right \}\\
&=&\theta_{2s,i}\#\left\{n: \left \lfloor \frac{N}{3^{8s}}\right \rfloor \cdot 3^{8s}+\left \lfloor \frac{N-\left \lfloor \frac{N}{3^{8s}}\right \rfloor \cdot 3^{8s}}{3^{6s}} \right \rfloor \cdot 3^{6s} \le n \le N\right \} + O(3^{6s}).
\end{eqnarray*}
It follows that 
\begin{eqnarray*}
&&\#\{n: l\le N, n \in E_{8s}(N), \tau(f(n)) = i\} \\
&=&\theta_{2s,i}\cdot  |A_{8s}(N)|\cdot 3^{8s} 
+ \theta_{2s,i}\left \lfloor \frac{N-\left \lfloor \frac{N}{3^{8s}}\right \rfloor \cdot 3^{8s}}{3^{6s}} \right \rfloor \cdot 3^{6s} \\
&&+\theta_{2s,i}\#\left\{n: \left \lfloor \frac{N}{3^{8s}}\right \rfloor \cdot 3^{8s}+\left \lfloor \frac{N-\left \lfloor \frac{N}{3^{8s}}\right \rfloor \cdot 3^{8s}}{3^{6s}} \right \rfloor \cdot 3^{6s} \le n \le N, \tau(f(n)) = i\right \}
+ O(3^{6s})\\
&=& \theta_{2s,i}(|A_{8s}(N)|\cdot 3^{8s} + \left \lfloor \frac{N-\left \lfloor \frac{N}{3^{8s}}\right \rfloor \cdot 3^{8s}}{3^{6s}} \right \rfloor \cdot 3^{6s}
\\
&&+\#\left \{n: \left \lfloor \frac{N}{3^{8s}}\right \rfloor \cdot 3^{8s}
+\left \lfloor \frac{N-\left \lfloor \frac{N}{3^{8s}}\right \rfloor \cdot 3^{8s}}{3^{6s}} \right \rfloor \cdot 3^{6s} \le n \le N, \tau(f(n)) = i \right \}) 
+ O(3^{6s})\\
&=&\theta_{2s,i}|E_{8s}(N)|+ O(3^{6s}) = \alpha_{i}|E_{8s}(N)|+ O(3^{4S+2s}).
\end{eqnarray*}
\end{proof}
Since $G(N) = O(\sqrt{N})$, one has
\[
\#\{n\le N: f(n)<3^{8}, \tau(f(n)) = i\}\le F(3^{8})\cdot G(N) = O(\sqrt{N})
\]
Then we obtain
\begin{eqnarray*}
\sum_{s=1}^{S}|E_{8s}(N)| = \#\{n\le N: 8\le H(f(n))\} = N-O(\sqrt{N}).
\end{eqnarray*}
Furthermore,
\begin{eqnarray*}
\#\{n: n\le N, \tau(f(n)) = i\} &=& O(\sqrt{N}) + \sum_{s=1}^{S}\#\{n: n\le N, 8s \le H(f(n)) < 8s+8, \tau(f(n)) = i\}\\
&=&O(\sqrt{N}) + \sum_{s=1}^{S}(\alpha_{i}|E_{8s}(N)| + O(3^{4S+2s}))\\
&=& O(\sqrt{N}) + \alpha_{i}\sum_{s=1}^{S}|E_{8s}(N)| + \sum_{s=1}^{S}O(3^{4S+2s}))\\
&=&O(\sqrt{N}) + \alpha_{i}(N+O(\sqrt{N})) + O(3^{6S}) = \alpha_{i}N + O(N^{3/4}), 
\end{eqnarray*}
which proves (1). 
\vskip3mm

Now we prove (2). Obviously, 
\begin{eqnarray*}
\#\{n\le N: \tau(f(n)) \neq \tau(f(n-1))\}
&=&\#\{n\le N: \tau(f(l)) \neq \tau(f(n-1)), H(f(n)) = H(f(n-1))\} \\
&+&\#\{n\le N: \tau(f(n)) \neq \tau(f(n-1)), H(f(n)) \neq H(f(n-1))\}.
\end{eqnarray*}

Now let $n = \sum_{i=0}^{8S+7}\delta_{i}3^{i}$, $\delta_{i}\in \{0,1,2\}$.
If $H(f(n)) \le 4S+4$, then $\delta_{4S+6} = 0$, $\delta_{4S+8} = 0, \dots{} ,\delta_{8S+6} = 0$. So we can choose only the remaining digits $\delta_{0}, \delta_{1}, \dots{}  ,\delta_{4S+5}, \dots{} ,\delta_{8S+7}$ which implies that
\begin{eqnarray*}
\#\{n: n\le N: H(f(n)) \le 4S+4\} \le 3^{6S+7} = O(N^{3/4}).
\end{eqnarray*}
\vskip2mm

\begin{proposition}
For $S/2 \le s \le S$, we have
\begin{eqnarray*}
\#\{n\le N: \tau(f(n)) \neq \tau(f(n-1)), 8s \le H(f(n)) = H(f(n-1))< 8s+8\} = O(N^{3/4}).
\end{eqnarray*}
\end{proposition}

\begin{proof}
 We prove that if $\tau(f(n)) \neq \tau(f(n-1)), 8s \le H(f(n)) = H(f(n-1))< 8s+8$, then $3^{2s}\mid n$. Assume that $3^{2s}\nmid n$. From now on let
\begin{eqnarray*}
 n = \sum_{i=0}^{8S+7}\delta_{i}3^{i}, \delta_{i}\in \{0,1,2\}
\hskip4mm \mbox{and}\hskip4mm 
 n-1 = \sum_{i=0}^{8S+7}\delta_{i}^{'}3^{i}, 
 \delta_{i}^{'}\in \{0,1,2\}.
 \end{eqnarray*} 
 Since
\begin{eqnarray*}
 n\not\equiv 0 \pmod {3^{2s}}, 
\end{eqnarray*} 
 there exists an integer $r$ with $0\le r<2s $ such that 
 \[ 
 \delta_r\ne 0. 
 \] 
 Therefore,
 we have $\delta_j = \delta_{j}^{'}$ for every $2s \le j \le 8S+7$ so that $(\delta_{2s}, \dots{} ,\delta_{6s-2}) = (\delta_{2s}^{'}, \dots{} ,\delta_{6s-2}^{'})$, which implies that 
 $\tau(f(n)) = \tau(f(n-1))$ a contradiction. It follows that
 \begin{eqnarray*}
 &&\#\{n\le N: \tau(f(n)) \neq \tau(f(n-1)), 8s \le H(f(n)) = H(f(n-1)) < 8s+8\} \\
&&\le \#\{n\le N: 8s \le H(f(n)) = H(f(n-1))< 8s+8, 3^{2s}\mid n\}.
\end{eqnarray*}
If $8s \le  H(f(n)) < 8s + 8$ and $3^{2s}\mid n$, then $\delta_j = 0$ for every $0 \le j \le 2s-1$ and  $\delta_{8s+8} = \delta_{8s+10} = \dots{} = \delta_{8S+6} = 0$, thus we can choose arbitrary $\delta_{2s}, \delta_{2s+1}, \dots{} ,\delta_{8s+7}, \delta_{8s+9}, \dots{} ,\delta_{8S+7}$ so that
\[
\#\{n\le N: \tau(f(n)) \neq \tau(f(n-1)), 8s \le H(f(n)) = H(f(n-1)) < 8s+8\} \le 3^{4S+2s+8} = O(3^{4S+2s}).
\]
It follows that 
\begin{eqnarray*}
&&\#\{n\le N: \tau(f(n)) \neq \tau(f(n-1)), H(f(n)) = H(f(n-1))\} \\
&\le& \#\{n\le N: H(f(n)) = H(f(n-1))\le 4S+4\}\\
&&+\sum_{S/2\le s \le S}\#\{n\le N: \tau(f(n)) \neq \tau(f(n-1)), 8s \le H(f(n)) = H(f(n-1))< 8s+8\}\\
&\le& O(N^{3/4}) + \sum_{S/2 \le s \le S}O(3^{4S+2s})\\
&=&O(N^{3/4}) + O(3^{6S}) = O(N^{3/4}).
\end{eqnarray*}
\end{proof}
Furthermore, 
\begin{eqnarray*}
&&\#\{n\le N: \tau(f(n)) \neq \tau(f(n-1)), H(f(n)) \neq H(f(n-1))\} \\
&=& \#\{n\le N: \tau(f(n)) \neq \tau(f(n-1)), H(f(n)) \neq H(f(n-1)), H(f(n))\le 4S+4\}\\
&+& 
\sum_{s=S/2}^{S}\#\{n\le N: \tau(f(n)) \neq \tau(f(n-1)), H(f(n)) \neq H(f(n-1)), 8s \le H(f(n)) < 8s+8\}.
\end{eqnarray*}
If $\tau(f(n)) \neq \tau(f(n-1)), H(f(n)) \neq H(f(n-1)), 8s \le H(f(n)) < 8s+8$ for $n\le N$, then $3^{8s}\mid n$, otherwise $\delta_j = \delta_j^{'}$ for every $j=8s, 8s+1, \dots{}$ so that $H(f(n)) = H(f(n-1))$ would give a contradiction. Thus we have
\begin{eqnarray*}
&&\#\{n\le N: \tau(f(n)) \neq \tau(f(n-1)), H(f(n)) \neq H(f(n-1))\}\\
&\le& \#\{n\le N: H(f(n))\le 4S+4\} 
+ \sum_{S/2 \le s \le S}\#\{n\le N: 3^{8s}\mid n\}\\
&\le& O(N^{3/4}) + \sum_{S/2 \le s \le S}\left(\frac{N}{3^{8s}} + 1\right)\\
&=& O(N^{3/4}) + O\left(\frac{N}{3^{4S}}\right) = O(N^{3/4}),
\end{eqnarray*}
which proves (2).
\end{proof}

\section{Proof of Main Result}    
Define the set $\mathcal{A}_{0}$ as
\begin{eqnarray*}
\mathcal{A}_{0} = A_{F} \cup A_{G} = \left(\bigcup_{f\in F}\{mf+\{0,1,\dots{} ,\tau(f)-1\}\}\right) \bigcup \left(\bigcup_{g\in G}\{mg + \{0,1,\dots{} ,m-1\}\}\right).
\end{eqnarray*}
We will prove that $\mathcal{A}_{0}$ is suitable. Now we compute the number of solutions of the equation $a+b=nm+r$, $a,b\in \mathcal{A}_{0}$.
\begin{proposition}
Let $n\ge 1$ be an integer such that $\tau(f(n)) = \tau(f(n-1)) = i$,
$[nm, (n+1)m-1] \cap (A_{F} + A_{F}) = \emptyset$, 
$[nm, (n+1)m-1] \cap (A_{G} + A_{G}) = \emptyset$. 
Then for every $0 \le r < m$, we have $r_2(\mathcal{A}_{0}, nm+r) = i$.
\end{proposition}

\begin{proof}
If $nm + r = (fm+c) + (gm+d)$, where $f\in F, g\in G$ and $0 \le c \le \tau(f)-1$, $0\le d \le m-1$, then $0 \le c+d \le 2m-2$. Thus, either $f+g=n$ and $c+d=r$, or $f+g=n-1$ and $c+d=m+r$.
\vskip2mm

For $0 \le r \le \tau(f(n))-2$, if $f+g=n$, $f\in F, g\in G$,
then $nm + r = (fm+j) + (gm+r-j)$ for every $0\le j \le r$. If 
$f+g=n-1$, $f\in F, g\in G$, then $nm + r = (fm+j) + (gm+m+r-j)$ for every $r+1\le j \le \tau(f(n))-1 = \tau(f(n-1))-1$. 
It follows that the number of solutions of $a+b =nm + r$ with $a\in A_F$, $b\in A_G$ is $\tau(f(n)) = \tau(f(n-1))=i$. On the other hand, there is no solution with $a,b\in A_F$ or $a,b\in A_G$ such that $r_2(\mathcal{A}_{0}, nm+r) = i$. 
\vskip2mm

If $\tau(f(n))-1 \le r \le m-1$, then the solutions of the equation $a+b=nm+r$ with $a\in A_F$, $b\in A_G$ are
$nm + r = (fm+j) + (gm+r-j)$ for every $0\le j \le \tau(f(n))-1$, where $f+g=n$, $f\in F, g\in G$. Similarly as above, we get that $r_2(\mathcal{A}_{0}, nm+r) = i$ for every $0\le r \le m-1$.
\end{proof}

It follows from Lemma 2.1 that for $0\le i\le m$, we have
\begin{eqnarray*}
\#\{n\le N/m: \tau(f(n)) = \tau(f(n-1)) = i\} = \alpha_{i}\frac{N}{m}+O(N^{3/4}).
\end{eqnarray*}
If $n\in F+F$, then $n = \sum_{i=0}^{\infty}\beta_{i}9^{i}$, 
$\beta_{i}\in \{0,1,2,3,4\}$ and $\beta_i = 0$ all but finitely many $i$. It follows that $(F+F)(9^{m}-1) = 5^{m}$. Therefore,  
\begin{eqnarray*}
(F+F)(x) = O(x^{\frac{\log 5}{\log 9}}) = O(x^{3/4}).
\end{eqnarray*}
Then
\begin{eqnarray*}
(A_F + A_F)(x) \le (F+F)(x)\cdot 2m = O(x^{3/4}),
\end{eqnarray*}
\begin{eqnarray*}
(A_G + A_G)(x) \le (G+G)(x)\cdot 2m \le (F+F)(x)\cdot 2m = O(x^{3/4}).
\end{eqnarray*}
It follows that 
\begin{eqnarray*}
\#\{n\le N/m: [nm, (n+1)m-1] \cap (A_{F} + A_{F})\neq \emptyset\} = O(N^{3/4})
\end{eqnarray*}
and
\begin{eqnarray*}
\#\{n\le N/m:[nm, (n+1)m-1] \cap (A_{G} + A_{G})\neq \emptyset\} = O(N^{3/4}).
\end{eqnarray*}
Thus,
\begin{eqnarray*}
\#\{n\le N/m:\tau(f(n)) = \tau(f(n-1)) = i,[nm, (n+1)m-1] \cap (A_{F} + A_{F}) = \emptyset, 
\end{eqnarray*}
\begin{eqnarray*}
[nm, (n+1)m-1] \cap (A_{G} + A_{G}) = \emptyset\} = \alpha_{i}\frac{N}{m} + O(N^{3/4}).
\end{eqnarray*}
Therefore,
\begin{eqnarray*}
&&\#\{n\le N: r_2(\mathcal{A}_{0},n) = i\} \\
&=& m\cdot \#\{n\le N/m:\tau(f(n)) = \tau(f(n-1)) = i,[nm, (n+1)m-1] \cap (A_{F} + A_{F}) = \emptyset, \\
&&[nm, (n+1)m-1] \cap (A_{G} + A_{G}) = \emptyset\} + O(1)\\
&&+ O(\#\{n\le N/m: \tau(f(n)) \neq \tau(f(
n-1))\}) + O(\#\{n: n\le N, n\in A_{F} 
+ A_{F}\}) \\
&&+ O(\#\{n: n\le N, n\in A_{G} + A_{G}\}) \\
&=&\alpha_{i}N + O(N^{3/4}) + O(1) + O(N^{3/4}) + O(N^{3/4}) + O(N^{3/4})\\
&=&\alpha_{i}N + O(N^{3/4}).
\end{eqnarray*}
It infers that there exists at least one desired set $\mathcal{A}_{0}$. If for some set $\mathcal{A}$ we have $\#(\mathcal{A}_{0}\bigtriangleup \mathcal{A})<\infty$ (the symmetric difference is finite), and $r_2(\mathcal{A}_{0},n)\ne r_2(\mathcal{A},n)$ for some nonnegative integer $n$, then there exists an integer $c\in \mathcal{A}_{0}\bigtriangleup \mathcal{A}$ and an integer  $d\in \mathcal{A}_{0}\cup \mathcal{A}$ such that $n=c+d$. Since $\# (\mathcal{A}_{0}\bigtriangleup \mathcal{A})<\infty$ and $(\mathcal{A}_{0}\cup \mathcal{A})(N)=O(\sqrt{N})$, we get that 
\begin{eqnarray*}
\# \{ n\le N: r_2(\mathcal{A}_{0},n)\ne r_2(\mathcal{A},n)\} =O(\sqrt{N}).
\end{eqnarray*}
It follows that 
\begin{eqnarray*}
\#\{n\le N: r_2(\mathcal{A},n) = i\} =\alpha_{i}N + O(N^{3/4})
\end{eqnarray*}
for every $0\le i\le m$. It follows that there exist infinitely many sets $\mathcal{A}\subseteq \mathbb{N}_0$ with the desired properties. 
\vskip3mm

This completes the proof of Theorem 1.1.
\vskip3mm

\bigskip

{\bf Use of AI disclaimer:} 
During the development of this work, the authors used OpenAI ChatGPT (GPT-5.5 Thinking and GPT-5.5 Pro), as an auxiliary tool. Based on the works \cite{Chen2016}, \cite{Sarkozy1997} and a probabilistic method introduced by the authors, one suggestion (the use of some kind of function $\tau(f(n))$) arises from this interaction. The authors independently verified, completed, and wrote the argument, and take full responsibility for all mathematical claims and the final content of the paper.

\end{document}